\documentclass[11pt]{article}

\usepackage{amsmath}
\usepackage{amsfonts}
\usepackage{amssymb}
\usepackage[hidelinks]{hyperref}
\newtheorem{theorem}{\bf Theorem}[section]
\newtheorem{lemma}[theorem]{\bf Lemma}
\newtheorem{proposition}[theorem]{\bf Proposition}
\newtheorem{corollary}[theorem]{\bf Corollary}

\newenvironment{proof}{\noindent{\em Proof:}}{\quad \hfill$\Box$\vspace{2ex}}
\newtheorem{remark}[theorem]{\bf Remark}

\numberwithin{equation}{section}

\def \e {{\,e}}
\def \dd {{\,d}}
\newcommand{\supp}{\operatorname{supp}}

\def \qquad {\quad\quad}

\makeatletter

\newcommand{\Rmnum}[1]{\expandafter\@slowromancap\romannumeral #1@}
\makeatother

\title{Exact-Support Counterexamples to Euclidean-to-Spherical Transfer of Positive Definiteness in Even Dimensions}
\author{Wentao Huang\thanks{Corresponding author. Department of Mathematics, The University of Hong Kong, Hong Kong, P.R. China. E-mail address: \textit{huangwt@hku.hk}.}
\and Haizhang Zhang\thanks{School of Mathematics, Sun Yat-sen University, Guangzhou, P.R. China. E-mail address: \textit{zhhaizh2@sysu.edu.cn}. Supported in part by the National Natural Science Foundation of China (Grant Nos. 12671642 and 12371103) and the Guangdong Basic and Applied Basic Research Foundation (Grant No. 2024A1515011194).}}
\date{}
\hypersetup{
    pdftitle={Exact-Support Counterexamples to Euclidean-to-Spherical Transfer of Positive Definiteness in Even Dimensions},
    pdfauthor={Wentao Huang and Haizhang Zhang},
    pdfsubject={Radial positive-definite functions on Euclidean spaces and spheres},
    pdfkeywords={positive definite function; sphere; compact support; radial Fourier transform; Gegenbauer polynomial}
}

\begin{document}

    \maketitle
	\begin{abstract}
       For every odd integer \(d\geq3\), a continuous function \(\varphi\colon[0,\infty)\to\mathbb R\) supported in \([0,\pi]\) and isotropic positive definite on \(\mathbb R^d\) remains so on \(\mathbb S^d\). In even dimensions, recent work shows that this transfer fails under every prescribed positive upper bound on the support. We prove an exact-support refinement with a construction uniform in the prescribed radius. More precisely, for each \(d=2m\geq2\) and \(R\in(0,\pi]\), we construct a function \(\varphi\) whose radial extension belongs to \(C_c^\infty(\mathbb R^d)\) and has support radius exactly \(R\), such that \(\varphi(\|\mathbf{x}-\mathbf{y}\|_2)\) is strictly positive definite on \(\mathbb R^d\), whereas \(\varphi(\rho(\mathbf{x},\mathbf{y}))\) is not positive definite on \(\mathbb S^d\). Thus every admissible support radius is attained by a smooth, strictly Euclidean positive-definite counterexample.
		
	\end{abstract}
	
	

    \section{Introduction}\label{sec:introduction}

    A natural way to construct an isotropic kernel on a sphere is to start with a radial positive-definite function on a Euclidean space. Compactly supported radial kernels are especially useful in approximation and spatial statistics, where localization produces sparse interpolation and covariance matrices \cite{Gneiting2002,Wendland1995}. There are, however, two different ways to transfer such a function to a sphere. Since
    \[
     \|\mathbf{x}-\mathbf{y}\|_2
     =2\sin\!\left(\frac{\rho(\mathbf{x},\mathbf{y})}{2}\right),
     \qquad \mathbf{x},\mathbf{y}\in\mathbb S^d,
    \]
    restricting a radial kernel on \(\mathbb R^{d+1}\) to the embedded sphere automatically gives a positive-definite kernel with radial function \(r\mapsto\varphi(2\sin(r/2))\). Direct transfer keeps the same function \(\varphi\) and simply replaces Euclidean distance by geodesic distance. This second operation has no general positivity principle behind it, and the distinction matters in covariance modeling on global domains \cite{Gneiting2013,HuangZhangRobeson2011}.
    
    More precisely, let
    \[
     \mathbb S^d
     =\{\mathbf{x}\in\mathbb R^{d+1}:\|\mathbf{x}\|_2=1\},
     \qquad
     \rho(\mathbf{x},\mathbf{y})
     =\arccos(\mathbf{x}\cdot\mathbf{y}),
    \]
    and let \(\Phi_d\) be the class of continuous functions \(\varphi\colon[0,\infty)\to\mathbb R\) such that \(\varphi(\|\mathbf{x}-\mathbf{y}\|_2)\) is positive definite on \(\mathbb R^d\). Similarly, let \(\Psi_d\) be the class of continuous functions \(\psi\colon[0,\pi]\to\mathbb R\) such that \(\psi(\rho(\mathbf{x},\mathbf{y}))\) is positive definite on \(\mathbb S^d\). Bochner's theorem describes \(\Phi_d\) through nonnegative Euclidean spectral measures \cite{Bochner1932,Bochner1933}, while Schoenberg's theorem describes \(\Psi_d\) through nonnegative Gegenbauer coefficients \cite{Schoenberg1942}. Direct transfer therefore asks whether the first spectral positivity condition forces the second under the support condition \(\supp \varphi \subseteq [0,\pi]\), where \(\pi\) is the maximal geodesic distance on \(\mathbb S^{d}\). We write this property as
    \begin{equation}\label{eq:transfer-question}
     (\mathcal T_d)\qquad
     \varphi\in\Phi_d,\quad
     \supp\varphi\subseteq [0,\pi]
     \quad\Longrightarrow\quad
     \varphi|_{[0,\pi]}\in\Psi_d.
    \end{equation}
    The first compact-support transfer theorem was established for radial functions positive definite on \(\mathbb R^{3}\): under the support condition, every such function induces a positive-definite kernel on \(\mathbb S^{j}\) for \(j=1,2,3\) \cite{Gneiting2013}. An extension to all odd dimensions was stated in \cite{NieMa2019}; Feng and Ge subsequently established the theorem by a different method \cite{FengGe2022}. A further refinement proved positivity of every resulting spherical spectral coefficient, hence strict positive definiteness \cite{Lu2025}. In even dimensions, P\'olya-type criteria and positivity results for Jacobi integrals produce important compactly supported families on spheres \cite{BeatsonCastellXu2014,FengGe2022,Lu2025,Xu2018}. These results impose additional structure on the function and do not settle \((\mathcal T_d)\) for the full class \(\Phi_d\).
    
    After obtaining the main result independently, the authors learned of Ge's contemporaneous work \cite{Ge2026}. Ge constructs smooth, strictly Euclidean positive-definite counterexamples below any prescribed support bound. Both proofs use the degree-two Fourier window and polynomial localization, but implement them differently. Ge derives a global window formula and uses polynomial approximation to concentrate Fourier mass in its negative region. The present proof instead uses compact frequency support, integrability, and the negative quadratic expansion at the origin. It applies the explicit filter \((I+\Delta/(d+1)^2)^M\) to a compactly supported radial bump. Real analyticity shows that filtering preserves the bump's support, and the Titchmarsh--Lions theorem gives exact support after the final convolution. Consequently, the present paper gives a complementary exact-support refinement: every prescribed radius \(R\in(0,\pi]\) is attained as the actual support radius, rather than serving only as an upper bound. For fixed even \(d\), the same sufficiently large \(M\) works for all \(R\).

    For a nonzero compactly supported function on \([0,\infty)\), we call \(\max\supp\varphi\) its support radius, and we write \(B_R=\{\mathbf{x}\in\mathbb R^d:\|\mathbf{x}\|_2<R\}\).
    
    \begin{theorem}[Counterexamples at every support radius]
     \label{thm:main}
     Let \(d=2m\), where \(m\geq1\). For every prescribed \(R\in(0,\pi]\), there is a function \(\varphi\colon[0,\infty)\to\mathbb R\) with the following properties:
     \begin{enumerate}
      \item the radial extension \(\varphi(\|\,\cdot\,\|_2)\) belongs to \(C_c^\infty(\mathbb R^d)\), and the support radius of \(\varphi\) is exactly \(R\);
      \item \(\varphi(\|\mathbf{x}-\mathbf{y}\|_2)\) is strictly positive definite on \(\mathbb R^d\);
      \item \(\varphi(\rho(\mathbf{x},\mathbf{y}))\) is not positive definite on \(\mathbb S^d\).
     \end{enumerate}
     Moreover, one may take
     \[
      \varphi(r)=(q*q)(r\mathbf{e}_1),\qquad \mathbf{e}_1=(1,0,\ldots,0)^{T},
     \]
     for a nonzero real radial \(q\in C_c^\infty(\mathbb R^d)\) satisfying \(\supp q=\overline B_{R/2}\), and the degree-two Gegenbauer coefficient of \(\varphi\) is strictly negative.
    \end{theorem}
    
    Theorem~\ref{thm:main} strengthens the known failure of \((\mathcal T_d)\) by prescribing the support exactly: the radius \(R\) is attained, rather than merely serving as an upper bound. The construction is also uniform in the prescribed radius: in each fixed even dimension, the exponent of the polynomial filter can be chosen independently of \(R\); see Lemma~\ref{lem:negative-energy}. Thus counterexamples occur at every radius not exceeding \(\pi\), including the endpoint, while retaining smoothness and strict Euclidean positive definiteness. This failure cannot be attributed to behavior near the maximal geodesic distance \(\pi\), insufficient regularity, or degeneracy of the Euclidean kernel. It is detected by the degree-two Gegenbauer coefficient, the first nonconstant even spherical mode. Together with the odd-dimensional theorem, Theorem~\ref{thm:main} recovers the dimension classification of direct transfer for \(d\geq2\) with exact control of the support.
    
    The proof separates into a spherical calculation and a support-preserving localization argument. For a self-convolution \(f=q*q\), the degree-two spherical coefficient becomes a Fourier pairing with the transform of
    \[
     W_d(\mathbf{x})
     =C_2^{(d-1)/2}(\cos\|\mathbf{x}\|_2)
      \left(\frac{\sin\|\mathbf{x}\|_2}{\|\mathbf{x}\|_2}\right)^{d-1},
    \]
    where \(C_2^{(d-1)/2}\) denotes the degree-two Gegenbauer polynomial with parameter \((d-1)/2\). The Fourier transform \(\widehat W_d\) is integrable and supported in the ball of radius \(d+1\), and near the origin its radial representation satisfies
    \[
     w_{2m}(k)=-\gamma_{2m}k^2+O(k^4),
     \qquad
     \gamma_{2m}=\frac{2m-1}{2(2m+1)(2m+1)!!}>0.
    \]
    Abel regularization and the anti-periodicity of the window reduce the sign of the quadratic term to a positive one-dimensional integral. A high power of \(I+\Delta/(d+1)^2\), applied to a bump of prescribed radius, then concentrates the relevant Fourier energy near the origin without enlarging the support. Self-convolution turns the filtered bump into a strictly positive-definite function and doubles its support radius, completing the construction. In each fixed dimension, the same power threshold works for all prescribed radii. This localization step uses only the compact support and integrability of \(\widehat W_d\), together with its negative leading term at the origin. Once these properties have been established, no further details of the Gegenbauer calculation enter the argument.

    The remainder of the paper is organized as follows. Section~\ref{sec:window} expresses the degree-two spherical coefficient as a Fourier pairing. Section~\ref{sec:negative-window} identifies the corresponding compactly supported Fourier window and establishes its negative low-frequency expansion. Section~\ref{sec:construction} develops the support-preserving Fourier localization argument and completes the proof of Theorem~\ref{thm:main}.

    \section{The degree-two coefficient as a Fourier pairing}
    \label{sec:window}
    
    Throughout this paper, we use the unitary Fourier transform
    \begin{equation}\label{eq:fourier-convention}
     \widehat h(\boldsymbol{\xi})
     =(2\pi)^{-d/2}\int_{\mathbb R^d}
       \e^{-i\mathbf{x}\cdot\boldsymbol{\xi}}h(\mathbf{x})\dd\mathbf{x},
     \qquad
     h(\mathbf{x})
     =(2\pi)^{-d/2}\int_{\mathbb R^d}
       \e^{i\mathbf{x}\cdot\boldsymbol{\xi}}
       \widehat h(\boldsymbol{\xi})\dd\boldsymbol{\xi}.
    \end{equation}
    For a function \(q\) on \(\mathbb R^d\), set \(\widetilde q(\mathbf{x})=\overline{q(-\mathbf{x})}\). Then
    \begin{equation}\label{eq:autocorrelation}
     \widehat{q*\widetilde q}(\boldsymbol{\xi})
     =(2\pi)^{d/2}|\widehat q(\boldsymbol{\xi})|^2.
    \end{equation}
    
    For the rest of the paper, set
    \begin{equation}\label{eq:parameters}
     d=2m,\qquad
     \nu=\frac d2-1=m-1,\qquad
     \lambda=\frac{d-1}{2}.
    \end{equation}
    By Schoenberg's theorem, a continuous isotropic function on \(\mathbb S^d\) is positive definite precisely when all its Gegenbauer coefficients are nonnegative \cite{Schoenberg1942}. Their normalizing constants are positive, so we use the unnormalized coefficients
    \begin{equation}\label{eq:gegenbauer-coefficient}
     B_{\ell,d}(\psi)
     =\int_0^\pi
       \psi(r)C_\ell^\lambda(\cos r)(\sin r)^{d-1}\dd r.
    \end{equation}
    Thus \(B_{\ell,d}(\psi)<0\) for one \(\ell\) implies \(\psi\notin\Psi_d\).
    
    We shall detect failure using only \(\ell=2\). Define the radial function
    \begin{equation}\label{eq:window}
     W_d(\mathbf{x})
     =C_2^\lambda(\cos\|\mathbf{x}\|_2)
      \left(
       \frac{\sin\|\mathbf{x}\|_2}{\|\mathbf{x}\|_2}
      \right)^{d-1},
     \qquad \mathbf{x}\in\mathbb R^d,
    \end{equation}
    where the quotient has its continuous value at the origin. Let \(\omega_{d-1}=|\mathbb S^{d-1}|\) denote the surface area of the unit sphere in \(\mathbb R^d\).
    
    \begin{lemma}[Fourier representation of the degree-two coefficient]\label{lem:window}
    Let \(q\in C_c^\infty(\mathbb R^d)\) be real and radial, \(\supp q\subseteq\overline B_{R/2}\), and \(R\leq\pi\). If \(f=q*\widetilde q=q*q\) and \(\varphi(r)=f(r\mathbf{e}_1)\), then
    \begin{equation}\label{eq:window-identity}
     \omega_{d-1}B_{2,d}(\varphi)
     =\int_{\mathbb R^d}f(\mathbf{x})W_d(\mathbf{x})\dd\mathbf{x}
     =(2\pi)^{d/2}
      \left\langle\widehat W_d,|\widehat q|^2\right\rangle.
    \end{equation}
    The Fourier transform \(\widehat W_d\) is initially understood as a tempered distribution. The pairing is well defined because \(|\widehat q|^2\) is a Schwartz function.
    \end{lemma}
    
    \begin{proof}
    Since \(\supp f\subseteq\overline B_{R}\) and \(R\le \pi\), polar coordinates and \eqref{eq:window} give
    \[
        \begin{aligned}
            \int_{\mathbb R^d}f(\mathbf{x})W_d(\mathbf{x})\dd\mathbf{x}
            &=\omega_{d-1}\int_0^\infty f(r\mathbf{e}_1)C_2^\lambda(\cos r)(\sin r)^{d-1}\dd r\\
            &=\omega_{d-1}\int_0^\pi\varphi(r)C_2^\lambda(\cos r)(\sin r)^{d-1}\dd r\\
            &=\omega_{d-1}B_{2,d}(\varphi).
        \end{aligned}
    \]
    The definition of the Fourier transform on tempered distributions and \eqref{eq:autocorrelation} yield
    \[
     \int_{\mathbb R^d}f(\mathbf{x})W_d(\mathbf{x}) \dd\mathbf{x}
     =\left\langle\widehat W_d,\widehat f(-\,\cdot)\right\rangle
     =(2\pi)^{d/2}
      \left\langle\widehat W_d,|\widehat q|^2\right\rangle.
    \]
    \end{proof}

    \section{A negative low-frequency window in even dimensions}
    \label{sec:negative-window}
    
    We now identify the Fourier transform of \(W_d\). Its compact frequency support will make localization possible, while its sign near the origin will force the degree-two coefficient to be negative. The function \(W_d\) is not absolutely integrable at infinity, so its radial Fourier formula cannot initially be interpreted as an ordinary integral. We first make precise the Abel regularization used near the origin and identify its locally uniform limit with the distributional Fourier transform.

    Set
    \begin{equation}\label{eq:Gd}
     y_d(r)=C_2^\lambda(\cos r),
     \qquad
     G_d(r)=y_d(r)(\sin r)^{d-1}.
    \end{equation}

    \begin{lemma}[Abel identification near the origin]
    \label{lem:abel-identification}
    For \(\varepsilon>0\), let
    \[
     W_{d,\varepsilon}(\mathbf{x})
     =\e^{-\varepsilon\|\mathbf{x}\|_2}W_d(\mathbf{x}),
     \qquad
     F_\varepsilon(s)
     =\int_0^\infty \e^{-\varepsilon r}G_d(r)\cos(sr)\dd r.
    \]
    Then \(W_{d,\varepsilon}\in L^1(\mathbb R^d)\) and \(W_{d,\varepsilon}\to W_d\) in \(\mathcal S'(\mathbb R^d)\). Moreover, for every \(0<K<1\),
    \[
     \widehat W_{d,\varepsilon}(\boldsymbol{\xi})
     \longrightarrow w_d(\|\boldsymbol{\xi}\|_2)
     \quad\text{uniformly for }\boldsymbol{\xi}\in\overline B_K,
    \]
    where
    \begin{align}
     w_d(k)
     &=c_\nu\int_{-1}^{1}(1-t^2)^{\nu-1/2}F(kt)\dd t,
        \label{eq:wd-local-representation}\\
     F(s)
     &=\frac{\displaystyle
       \int_0^\pi G_d(r)\cos\!\bigl(s(r-\pi/2)\bigr)\dd r}
       {2\cos(\pi s/2)},
       \qquad |s|<1,                                      
       \label{eq:one-period-abel}\\
     c_\nu&=\frac1{2^\nu\sqrt\pi\,\Gamma(\nu+1/2)}.     \notag
    \end{align}
    Consequently, the restriction of \(\widehat W_d\) to \(B_1\) is the regular distribution induced by the real radial function \(w_d(\|\boldsymbol{\xi}\|_2)\), and \(w_d\) is even and real analytic on \((-1,1)\).
    \end{lemma}

    \begin{proof}
    The function \(W_d\) is smooth and bounded on \(\mathbb R^d\), and therefore defines a tempered distribution. For every \(\varepsilon>0\), the exponential cutoff makes \(W_{d,\varepsilon}\) integrable, and dominated convergence against Schwartz functions gives \(W_{d,\varepsilon}\to W_d\) in \(\mathcal S'(\mathbb R^d)\) as \(\varepsilon\to 0^{+}\). Continuity of the Fourier transform on \(\mathcal S'\) then gives
    \begin{equation}\label{eq:abel-distributional-convergence}
     \widehat W_{d,\varepsilon}\longrightarrow\widehat W_d
     \quad\text{in }\mathcal S'(\mathbb R^d).
    \end{equation}

    The radial Fourier formula \cite[Chapter~IV, Section~3]{SteinWeiss1971} and Poisson's integral for \(J_\nu\) \cite[Formula~(5.10.3)]{Lebedev1972} give, with absolutely convergent integrals,
    \begin{equation}\label{eq:regularized-transform}
     \begin{aligned}
      \widehat W_{d,\varepsilon}(k\mathbf e_1)
      &=\int_0^\infty
        \frac{J_\nu(kr)}{(kr)^\nu}\e^{-\varepsilon r}G_d(r)\dd r\\
      &=c_\nu\int_{-1}^{1}(1-t^2)^{\nu-1/2}F_\varepsilon(kt)\dd t.
     \end{aligned}
    \end{equation}
    Here the value at \(k=0\) is understood by continuity. Since \(G_d(r+j\pi)=(-1)^jG_d(r)\), splitting the integral into intervals of length \(\pi\) gives, for \(\varepsilon>0\),
    \[
        \begin{aligned}
         \int_0^\infty \e^{-(\varepsilon-is)r}G_d(r)\dd r
         &=\sum_{j=0}^\infty
           \int_{j\pi}^{(j+1)\pi}
           \e^{-(\varepsilon-is)r}G_d(r)\dd r\\
         &=\sum_{j=0}^\infty
           (-1)^j\e^{-(\varepsilon-is)j\pi}
           \int_0^\pi \e^{-(\varepsilon-is)r}G_d(r)\dd r\\
         &=\frac{\displaystyle
           \int_0^\pi \e^{-(\varepsilon-is)r}G_d(r)\dd r}
           {1+\e^{-\pi(\varepsilon-is)}}.
        \end{aligned}
    \]
    The geometric series converges absolutely because \(\e^{-\pi\varepsilon}<1\). On every compact subinterval of \((-1,1)\), the numerator converges uniformly as \(\varepsilon\to0^{+}\), while the denominator stays uniformly away from zero. Taking real parts therefore gives \(F_\varepsilon\to F\) locally uniformly, with \(F\) as in \eqref{eq:one-period-abel}. It follows from \eqref{eq:regularized-transform} that, for \(0<K<1\),
    \[
     \sup_{|k|\leq K}
     \left|\widehat W_{d,\varepsilon}(k\mathbf e_1)-w_d(k)\right|
     \leq c_\nu\left\|(1-t^2)^{\nu-1/2}\right\|_{L^1(-1,1)}
     \sup_{|s|\leq K}|F_\varepsilon(s)-F(s)|
     \longrightarrow0.
    \]
    Finally, if \(\chi\in C_c^\infty(B_K)\), local uniform convergence and \eqref{eq:abel-distributional-convergence} give
    \[
     \langle\widehat W_d,\chi\rangle
     =\lim_{\varepsilon\to0^{+}}
       \int_{\mathbb R^d}\widehat W_{d,\varepsilon}(\boldsymbol{\xi})
       \chi(\boldsymbol{\xi})\dd\boldsymbol{\xi}
     =\int_{\mathbb R^d}w_d(\|\boldsymbol{\xi}\|_2)
       \chi(\boldsymbol{\xi})\dd\boldsymbol{\xi}.
    \]
    Since \(K<1\) is arbitrary and every test function in \(C_c^\infty(B_1)\) is supported in some \(B_K\), this proves that the restriction of \(\widehat W_d\) to \(B_1\) is the regular distribution induced by \(w_d(\|\cdot\|_2)\). Equations \eqref{eq:wd-local-representation} and \eqref{eq:one-period-abel}, with differentiation under the integral on compact subintervals, also show that \(w_d\) is real, even, and analytic there.
    \end{proof}
    
    The lemma resolves the limiting step and provides an ordinary analytic representative of \(\widehat W_d\) near the origin. We now determine the global support and integrability of the transform and compute the first nonzero term of this representative.

    \begin{proposition}[The Fourier window]\label{prop:fourier-window}
    For every even \(d\geq2\), the distribution \(\widehat W_d\) is represented by a real radial function in \(L^1(\mathbb R^d)\), and
    \[
     \supp\widehat W_d\subseteq\overline B_{d+1}.
    \]
    Writing \(\widehat W_d(\boldsymbol{\xi})=w_d(\|\boldsymbol{\xi}\|_2)\) almost everywhere, \(w_d\) has a real even analytic representative on \((-1,1)\), and
    \begin{equation}\label{eq:local-expansion}
     w_d(k)=-\gamma_dk^2+O(k^4)
     \quad(k\to0),
     \qquad
     \gamma_{2m}
     =\frac{2m-1}{2(2m+1)(2m+1)!!}.
    \end{equation}
    In particular, \(w_d(k)<0\) for all sufficiently small nonzero \(k\).
    \end{proposition}
    
    \begin{proof}
    We divide the proof into three parts: compact frequency support, integrability of the Fourier transform, and its local expansion at the origin.

    \emph{Frequency support.}
    The Gegenbauer generating function in \cite{Lebedev1972}, Section~5.12, Formula~(5.12.7), gives
    \[
     C_2^\lambda(x)=-\lambda+2\lambda(\lambda+1)x^2,
    \]
    and hence
    \begin{equation}\label{eq:C2-trigonometric}
     C_2^\lambda(\cos r)
     =\lambda^2+\lambda(\lambda+1)\cos(2r).
    \end{equation}
    Thus the scalar function \(C_2^\lambda(\cos r)(\sin r/r)^{d-1}\) is even and entire. Its power series in \(r^2\) defines an entire extension to \(\mathbb C^d\) as a function of \(z_1^2+\cdots+z_d^2\). If \(\zeta^2=z_1^2+\cdots+z_d^2\), then
    \begin{equation}\label{eq:entire-bound}
     |\operatorname{Im}\zeta|
     \leq\|\operatorname{Im}\mathbf{z}\|_2,
     \qquad
     |W_d(\mathbf{z})|
     \leq C_d(1+\|\mathbf{z}\|_2)^2
           \exp\!\bigl((d+1)\|\operatorname{Im}\mathbf{z}\|_2\bigr).
    \end{equation}
    The two scalar factors have exponential types \(2\) and \(d-1\), respectively. The Paley--Wiener--Schwartz theorem therefore represents \(W_d\) as the Fourier--Laplace transform of a distribution supported in \(\overline B_{d+1}\) \cite[Theorem~7.3.1]{Hormander1990}. Equivalently,
    \begin{equation}\label{eq:window-support}
     \supp\widehat W_d\subseteq\overline B_{d+1}.
    \end{equation}
    
    \emph{Integrability.}
    Suppose first that \(d\geq4\). Since \(C_2^\lambda\) is bounded on
    \([-1,1]\),
    \[
     |W_d(\mathbf{x})|
     \leq C_d(1+\|\mathbf{x}\|_2)^{-(d-1)}.
    \]
    Thus \(W_d\in L^2(\mathbb R^d)\). By Plancherel, \(\widehat W_d\in L^2(\mathbb R^d)\), and the support inclusion \eqref{eq:window-support} and the Cauchy--Schwarz inequality give
    \[
     \|\widehat W_d\|_{L^1}
     \leq |B_{d+1}|^{1/2}\|\widehat W_d\|_{L^2}<\infty.
    \]
    
    For \(d=2\), one has \(\lambda=1/2\) and
    \begin{equation}\label{eq:W2-sine-form}
     W_2(r\mathbf e_1)=\frac{3\sin(3r)-\sin r}{8r}.
    \end{equation}
    With our normalization, the radial Fourier transform in two dimensions is the order-zero Hankel transform. Interpreted by Abel regularization, the classical discontinuous Bessel integral
    \cite[Formula~6.671.7]{GradshteynRyzhik2014} gives, for \(a,k>0\) and
    \(a\ne k\),
    \[
     \int_0^\infty J_0(kr)\sin(ar)\dd r
     =\frac{\mathbf 1_{\{k<a\}}}{\sqrt{a^2-k^2}}.
    \]
    This is also an identity of tempered distributions. Indeed, its right-hand side is integrable on \(\mathbb R^2\), and the inverse radial transform is
    \[
     \int_0^a\frac{kJ_0(kr)}{\sqrt{a^2-k^2}}\dd k
     =\frac{\sin(ar)}r,
    \]
    as follows directly by integrating the power series of \(J_0\) term by term. Thus Lemma~\ref{lem:abel-identification} identifies the Abel limit with this regular distribution. Applying the identity to the two terms in
    \eqref{eq:W2-sine-form} gives
    \begin{equation}\label{eq:W2-transform}
     \widehat W_2(k\mathbf e_1)
     =\frac18\left(
       \frac{3\,\mathbf 1_{[0,3)}(k)}{\sqrt{9-k^2}}
       -\frac{\mathbf 1_{[0,1)}(k)}{\sqrt{1-k^2}}
      \right).
    \end{equation}
    Since
    \[
     \int_0^a\frac{k}{\sqrt{a^2-k^2}}\dd k=a,
    \]
    polar coordinates yield
    \[
     \|\widehat W_2\|_{L^1(\mathbb R^2)}
     \le
     \frac{2\pi}{8}\left(3\int_{0}^{3}\frac{k}{\sqrt{9-k^2}}\dd k
       +\int_{0}^1\frac{k}{\sqrt{1-k^2}}\dd k\right)
     =\frac{5\pi}{2}.
    \]
    
    \emph{The sign at the origin.}
    Lemma~\ref{lem:abel-identification} gives the analytic representative \eqref{eq:wd-local-representation} of \(\widehat W_d\) on \(B_1\). Let
    \begin{equation}\label{eq:Jd-definition}
     \mathcal J_d=\int_0^\pi r^2G_d(r)\dd r.
    \end{equation}
    Gegenbauer orthogonality gives
    \[
     \int_0^\pi G_d(r)\dd r
     =\int_0^{\pi}C_2^\lambda(\cos r)(\sin r)^{d-1}\dd r=0.
    \]
    Moreover, \(G_d(\pi-r)=G_d(r)\), so
    \(\int_0^\pi(r-\pi/2)G_d(r)\dd r=0\). It follows that
    \begin{equation}\label{eq:centered-moment}
     \int_0^\pi (r-\pi/2)^2G_d(r)\dd r
     =\int_0^\pi r^2G_d(r)\dd r
     =\mathcal J_d.
    \end{equation}
    Expanding the numerator in \eqref{eq:one-period-abel} at the origin and using \eqref{eq:centered-moment}, we find
    \[
     \int_0^\pi G_d(r)\cos\!\bigl(s(r-\pi/2)\bigr)\dd r
     =-\frac{\mathcal J_d}{2}s^2+O(s^4).
    \]
    Since \([2\cos(\pi s/2)]^{-1}=1/2+O(s^2)\), this gives
    \begin{equation}\label{eq:abel-low-frequency}
     F(s)=-\frac{\mathcal J_d}{4}s^2+O(s^4)
     \qquad(s\to0).
    \end{equation}
    
    After substituting \(s=kt\), the expansion in \eqref{eq:abel-low-frequency} is uniform for \(-1\leq t\leq1\) when \(k\) is sufficiently small. Since
    \[
     \int_{-1}^{1}t^2(1-t^2)^{\nu-1/2}\dd t
     =\frac{\sqrt\pi\,\Gamma(\nu+1/2)}{2\Gamma(\nu+2)},
    \]
    equations \eqref{eq:wd-local-representation} and \eqref{eq:abel-low-frequency} give
    \begin{equation}\label{eq:wd-low-frequency}
     w_d(k)
     =-\frac{\mathcal J_d}{2^{\nu+3}\Gamma(\nu+2)}k^2+O(k^4).
    \end{equation}
    
    It remains to evaluate \(\mathcal J_d\). After the change of variables \(x=\cos r\), the degree-two Gegenbauer equation becomes
    \[
     \left((\sin r)^{d-1}y_d'(r)\right)'
     +2(d+1)(\sin r)^{d-1}y_d(r)=0.
    \]
    Substituting this identity into \eqref{eq:Jd-definition} and integrating by parts gives
    \[
     \mathcal J_d
     =\frac1{d+1}\int_0^\pi
      r(\sin r)^{d-1}y_d'(r)\dd r.
    \]
    The boundary terms vanish. Since \(y_d'(r)=-4\lambda(\lambda+1)\sin r\cos r\), one further integration by parts yields
    \begin{equation}\label{eq:Jd}
     \mathcal J_d
     =\frac{4\lambda(\lambda+1)}{(d+1)^2}
       \int_0^\pi(\sin r)^{d+1}\dd r
     =\frac{\lambda}{\lambda+1}
      \frac{\sqrt\pi\,\Gamma(\lambda+3/2)}
           {\Gamma(\lambda+2)}
     >0.
    \end{equation}
    Equations \eqref{eq:wd-low-frequency} and \eqref{eq:Jd} prove \eqref{eq:local-expansion}, initially with
    \begin{equation}\label{eq:gamma-intermediate}
     \gamma_d
     =\frac{\mathcal J_d}
            {2^{\nu+3}\Gamma(\nu+2)}.
    \end{equation}
    For \(d=2m\), the half-integer gamma identity
    \[
     \Gamma(m+3/2)
     =\frac{(2m+1)!!}{2^{m+1}}\sqrt\pi
    \]
    reduces \eqref{eq:gamma-intermediate} to
    \[
     \gamma_{2m}
     =\frac{2m-1}{2(2m+1)(2m+1)!!}.
    \]
    Since \(\mathcal J_d>0\), \eqref{eq:wd-low-frequency} shows that \(w_d(k)<0\) whenever \(0<|k|<\delta_d\) for some \(\delta_d>0\). This proves the proposition.
    \end{proof}

    \begin{remark}[Why degree two?]\label{rem:degree-two}
    \upshape
        The choice \(\ell=2\) is dictated by the low-frequency sign. For degree zero, \(C_0^\lambda=1\), so
        \[
         W_{0,d}(\mathbf{x})
         =\left(\frac{\sin r}{r}\right)^{d-1},
         \qquad r=\|\mathbf{x}\|_2.
        \]
        The same Abel argument as in Lemma~\ref{lem:abel-identification} gives an analytic representative of \(\widehat W_{0,d}\) near the origin, with
        \[
         \widehat W_{0,d}(\mathbf{0})
         =\frac{1}{2^{\nu+1}\Gamma(\nu+1)}
           \int_0^\pi(\sin r)^{d-1}\dd r
         >0.
        \]
        Thus the degree-zero window is positive near the origin. The first nonconstant case is degree one. Its window is
        \[
         W_{1,d}(\mathbf{x})
         =C_1^\lambda(\cos r)
          \left(\frac{\sin r}{r}\right)^{d-1},
         \qquad r=\|\mathbf{x}\|_2.
        \]
        Since \(C_1^\lambda(t)=2\lambda t=(d-1)t\),
        \[
         (d-1)\cos r\,(\sin r)^{d-1}
         =\frac{d-1}{d}\frac{\dd}{\dd r}(\sin r)^d.
        \]
        The same Abel argument as in Lemma~\ref{lem:abel-identification} gives an analytic representative of \(\widehat W_{1,d}\) near the origin. Integration by parts, followed by averaging the \(\pi\)-periodic function \((\sin r)^d\), gives
        \[
         \widehat W_{1,d}(\mathbf{0})
         =\frac{d-1}{d\pi\,2^\nu\Gamma(\nu+1)}
           \int_0^\pi(\sin r)^d\dd r
         >0.
        \]
        Thus both lower-degree windows are positive near the origin. By contrast, Proposition~\ref{prop:fourier-window} gives \(w_d(k)=-\gamma_dk^2+O(k^4)\). Consequently, degree two is the first spherical mode whose window has the negative low-frequency sign required by the localization argument.
    \end{remark}
    
    \section{Support-preserving localization and the counterexample}
    \label{sec:construction}
    
    Proposition~\ref{prop:fourier-window} produces a negative spectral band, but a sharp frequency cutoff would destroy compact support in physical space. We instead use a polynomial differential filter whose degree is chosen sufficiently large, depending only on the dimension.

    \subsection{A differential filter with exact support}
    
    Let
    \begin{equation}\label{eq:standard-bump}
     \eta(\mathbf{x})=c_d
     \begin{cases}
      \exp\!\left(-\dfrac{1}{1-\|\mathbf{x}\|_2^2}\right),
          &\|\mathbf{x}\|_2<1,\\
      0,  &\|\mathbf{x}\|_2\geq1,
     \end{cases}
    \end{equation}
    where \(c_d>0\) is chosen so that
    \begin{equation}\label{eq:bump-normalization}
     \widehat\eta(0)=1,
     \qquad\text{equivalently}\qquad
     \int_{\mathbb R^d}\eta(\mathbf{x})\dd\mathbf{x}=(2\pi)^{d/2}.
    \end{equation}
    Then \(\eta\) is nonnegative, real, radial, belongs to \(C_c^\infty(\mathbb R^d)\), is real analytic in \(B_1\), and has support \(\overline B_1\). For \(R>0\), set
    \begin{equation}\label{eq:scaled-bump}
     \eta_R(\mathbf{x})=R^{-d}\eta(\mathbf{x}/R).
    \end{equation}
    For \(M\in\mathbb N\), define
    \begin{equation}\label{eq:filter}
     p_M(s)=\left(1-\frac{s}{(d+1)^2}\right)^M,
     \qquad
     q_{M,R}
     =\left(I+\frac{\Delta}{(d+1)^2}\right)^M\eta_R.
    \end{equation}
    Then
    \begin{equation}\label{eq:filter-transform}
     \widehat q_{M,R}(\boldsymbol{\xi})
     =p_M(\|\boldsymbol{\xi}\|_2^2)
      \widehat\eta(R\boldsymbol{\xi}),
     \qquad
     \widehat q_{M,R}(0)=1.
    \end{equation}
    
    Since differentiation is local,
    \(\supp q_{M,R}\subseteq\overline B_R\). In fact,
    \begin{equation}\label{eq:exact-q-support}
     \supp q_{M,R}=\overline B_R.
    \end{equation}
    Indeed, \(q_{M,R}\) is real analytic in the connected open ball \(B_R\). If it vanished on a nonempty open subset of \(B_R\), the identity theorem for real-analytic functions on connected domains would imply that it vanishes identically in \(B_R\). Since it is continuous and supported in \(\overline B_R\), it would then vanish throughout \(\mathbb R^d\). This contradicts \eqref{eq:filter-transform}. Its nonzero set is therefore dense in \(B_R\), proving \eqref{eq:exact-q-support}.

    \subsection{Concentrating energy in the negative band}
    
    \begin{lemma}[Uniform negative energy]\label{lem:negative-energy}
    For every even \(d\geq2\), there is \(M_d\in\mathbb N\) such that
    \begin{equation}\label{eq:fixed-radius-energy}
     \left\langle\widehat W_d,
      |\widehat q_{M,R}|^2\right\rangle<0
    \end{equation}
    for every \(0<R\leq\pi/2\) and every \(M\geq M_d\).
    \end{lemma}
    
    \begin{proof}
    By Proposition~\ref{prop:fourier-window}, the pairing in \eqref{eq:fixed-radius-energy} is the ordinary integral
    \begin{equation}\label{eq:filtered-energy-integral}
     E_{M,R}
     =\int_{\mathbb R^d}w_d(\|\boldsymbol{\xi}\|_2)
       \left(p_M(\|\boldsymbol{\xi}\|_2^2)\right)^2
       |\widehat\eta(R\boldsymbol{\xi})|^2\dd\boldsymbol{\xi}.
    \end{equation}
    By the local expansion and the continuity of \(\widehat\eta\) at the origin, choose \(0<\delta_d<1\) so that
    \begin{equation}\label{eq:negative-band}
     w_d(k)\leq-\frac{\gamma_d}{2}k^2
     \quad(0<k<2\delta_d),
     \qquad
     |\widehat\eta(\boldsymbol{\xi})|^2\geq\frac12
     \quad\left(\|\boldsymbol{\xi}\|_2\leq\frac{\pi\delta_d}{2}\right).
    \end{equation}
    For \(M\) large enough that \(a_M:=(d+1)/(2\sqrt M)<\delta_d\), Bernoulli's inequality gives
    \[
     \left(p_M(k^2)\right)^2
     \geq1-\frac{2Mk^2}{(d+1)^2}\geq\frac12
     \qquad(0\leq k\leq a_M).
    \]
    If \(0<R\leq\pi/2\) and \(0\leq k\leq\delta_d\), then \(Rk\leq\pi\delta_d/2\), so the second inequality in \eqref{eq:negative-band} applies to \(\widehat\eta(Rk\mathbf{e}_1)\). Thus, on \(B_{a_M}\),
    \[
     w_d(k)\left(p_M(k^2)\right)^2|\widehat\eta(R\boldsymbol{\xi})|^2
     \leq-\frac{\gamma_d}{8}k^2.
    \]
    Since \(w_d<0\) on \((0,2\delta_d)\), polar coordinates and \eqref{eq:negative-band} yield, uniformly in \(R\),
    \begin{equation}\label{eq:local-energy-estimate}
     \int_{B_{\delta_d}}w_d(\|\boldsymbol{\xi}\|_2)
      \left(p_M(\|\boldsymbol{\xi}\|_2^2)\right)^2
      |\widehat\eta(R\boldsymbol{\xi})|^2\dd\boldsymbol{\xi}
     \leq-\frac{\omega_{d-1}\gamma_d}{8}
      \int_0^{a_M}k^{d+1}\dd k
     =-C_dM^{-(d+2)/2},
    \end{equation}
    where
    \[
     C_d
     =\frac{\omega_{d-1}\gamma_d}{8(d+2)}
     \left(\frac{d+1}{2}\right)^{d+2}>0.
    \]
    
    The normalization and nonnegativity of \(\eta\) imply
    \(|\widehat\eta|\leq1\). Set
    \[
     L_d=\|\widehat W_d\|_{L^1(\mathbb R^d)},
     \qquad
     \varrho_d=1-\frac{\delta_d^2}{(d+1)^2}<1.
    \]
    On \(\overline B_{d+1}\setminus B_{\delta_d}\), one has \(|p_M(\|\boldsymbol{\xi}\|_2^2)|\leq\varrho_d^M\). Hence
    \begin{equation}\label{eq:tail-L1-estimate}
     \left|\int_{\mathbb R^d\setminus B_{\delta_d}}
       w_d(\|\boldsymbol{\xi}\|_2)
       \left(p_M(\|\boldsymbol{\xi}\|_2^2)\right)^2
       |\widehat\eta(R\boldsymbol{\xi})|^2\dd\boldsymbol{\xi}\right|
     \leq L_d\varrho_d^{2M}.
    \end{equation}
    Combining \eqref{eq:filtered-energy-integral}, \eqref{eq:local-energy-estimate}, and \eqref{eq:tail-L1-estimate}, we obtain
    \[
     E_{M,R}
     \leq -C_dM^{-(d+2)/2}+L_d\varrho_d^{2M}.
    \]
    Since \(0<\varrho_d<1\), the second term decays exponentially and satisfies
    \[
     L_d\varrho_d^{2M}=o\bigl(M^{-(d+2)/2}\bigr)
     \qquad(M\to\infty).
    \]
    Hence \(E_{M,R}<0\) for every sufficiently large \(M\), uniformly in \(R\). This proves the lemma.
    \end{proof}

    \begin{remark}[Support-preserving spectral localization]
    \upshape
    The proof of Lemma~\ref{lem:negative-energy} is a general localization principle. It uses only three properties of the Fourier window: it is integrable, it has compact support, and its first nonzero term at the origin is negative. More generally, if a radial window satisfies
    \[
     w(k)=-c k^{2j}+O(k^{2j+2}),\qquad c>0,
    \]
    the same filter produces a negative local contribution of order \(M^{-(d+2j)/2}\), while the contribution away from the origin remains exponentially small. Thus any such negative low-frequency behavior can be detected by a polynomial in the Laplacian without enlarging physical support; the mechanism is independent of the special Gegenbauer calculation in Proposition~\ref{prop:fourier-window}.
    \end{remark}
    
    \subsection{Construction of the counterexample}

    We now combine the preceding results to construct the required counterexample and prove Theorem~\ref{thm:main}.
    
    \begin{proof}
    Fix the prescribed support radius \(R\in(0,\pi]\), choose any \(M\geq M_d\) as in Lemma~\ref{lem:negative-energy}, and set \(q=q_{M,R/2}\). Equations \eqref{eq:filter}--\eqref{eq:exact-q-support} show that \(q\) is real, radial, smooth, satisfies \(\widehat q(0)=1\), and has
    \[
     \supp q=\overline B_{R/2}.
    \]
    Define \(f=q*\widetilde q=q*q\) and \(\varphi(r)=f(r\mathbf e_1)\) for \(r\geq0\). Since \(\widehat f=(2\pi)^{d/2}|\widehat q|^2\) and \(\widehat q(0)=1\), the spectral density of \(f\) is nonnegative and positive on a neighborhood of the origin. Indeed, for distinct points \(\mathbf x_1,\ldots,\mathbf x_n\) and coefficients \(c_1,\ldots,c_n\), not all zero,
    \[
     \sum_{j,k=1}^n c_j\overline{c_k}f(\mathbf x_j-\mathbf x_k)
     =(2\pi)^{-d/2}\int_{\mathbb R^d}\widehat f(\boldsymbol\xi)
       \left|\sum_{j=1}^n c_j\e^{i\mathbf x_j\cdot\boldsymbol\xi}\right|^2
       \dd\boldsymbol\xi>0,
    \]
    because the exponential polynomial in the integrand is not identically zero and therefore cannot vanish on a nonempty open set. Hence \(f\) is strictly positive definite on \(\mathbb R^d\).

    Lemma~\ref{lem:negative-energy} gives
    \begin{equation}\label{eq:negative-filtered-energy}
     \left\langle\widehat W_d,|\widehat q|^2\right\rangle<0.
    \end{equation}
    Moreover, \(\supp f\subseteq\overline B_R\). The Titchmarsh--Lions theorem of supports for compactly supported distributions \cite[Theorem~4.3.3]{Hormander1990} gives
    \begin{equation}\label{eq:exact-autocorrelation-radius}
     \operatorname{conv}\supp f
     =\operatorname{conv}\supp q
      +\operatorname{conv}\supp\widetilde q
     =\overline B_{R/2}+\overline B_{R/2}
     =\overline B_R.
    \end{equation}
    Since \(f\) is radial, \eqref{eq:exact-autocorrelation-radius} shows that the support radius of \(\varphi\) is exactly \(R\). Since \(R\le\pi\), Lemma~\ref{lem:window} applies. Its Fourier representation and \eqref{eq:negative-filtered-energy} give
    \begin{equation}\label{eq:negative-B2}
     B_{2,d}(\varphi)<0,
    \end{equation}
    and Schoenberg's characterization implies \(\varphi\notin\Psi_d\).

    Finally, the radial extension of \(\varphi\) is \(f\in C_c^\infty(\mathbb R^d)\). The preceding arguments establish all three assertions of Theorem~\ref{thm:main}.
    \end{proof}
    
    \subsection{Consequences and scope of the result}

    We conclude this section by recording several consequences of the construction and clarifying the scope of the result. Together with the odd-dimensional theorem, Theorem~\ref{thm:main} recovers the classification of direct transfer by dimension. The remarks that follow identify where the distinction between odd and even dimensions enters, record the uniformity in the prescribed support radius, and indicate what forms of positive transfer may still remain possible.
    
    \begin{corollary}[Classification by dimension]\label{cor:classification}
    For every integer \(d\geq2\), the direct transfer property \((\mathcal T_d)\) holds if and only if \(d\) is odd.
    \end{corollary}

    \begin{proof}
    The odd-dimensional implication is proved in \cite{FengGe2022}. For every even \(d\), Theorem~\ref{thm:main} supplies a counterexample with any prescribed support radius in \((0,\pi]\).
    \end{proof}

    \begin{remark}[Why even dimensions differ]
    \upshape
        The distinction between odd and even dimensions enters through the analysis of the Fourier window, rather than through the localization argument. Since \(C_2^\lambda\) is even,
        \[
         G_d(r+\pi)=(-1)^{d-1}G_d(r).
        \]
        Moreover, \(G_d(\pi-r)=G_d(r)\). Hence, if
        \[
         A_d(s)=\int_0^\pi
         G_d(r)\cos\!\bigl(s(r-\pi/2)\bigr)\dd r,
        \]
        then \(A_d\) is real and even. For even \(d\), the anti-periodicity of \(G_d\) gives
        \[
         F(s)=\frac{A_d(s)}{2\cos(\pi s/2)}.
        \]
        Gegenbauer orthogonality and the moment calculation in Proposition~\ref{prop:fourier-window} then yield the negative quadratic term at the origin. For odd \(d\), periodicity gives instead
        \[
         \lim_{\varepsilon\downarrow0}
         \int_0^\infty \e^{-(\varepsilon-is)r}G_d(r)\dd r
         =
         \frac{iA_d(s)}{2\sin(\pi s/2)},
         \qquad 0<|s|<1,
        \]
        which is purely imaginary. Thus its cosine Abel sum vanishes near the origin. The dimensional distinction is therefore already encoded in the low-frequency behavior of the corresponding Fourier window; it is not introduced by the subsequent choice of bump or differential filter.
    \end{remark}

    \begin{remark}[Uniformity in the support radius]
    \upshape
    The integer \(M_d\) in Lemma~\ref{lem:negative-energy} is independent of \(R\in(0,\pi/2]\). Hence, in a fixed even dimension, one polynomial in the Laplacian works for the entire scaled bump family. After self-convolution, the support radius is exactly the prescribed value, rather than merely bounded by it.
    \end{remark}

    \begin{remark}[What remains possible]
    \upshape
    Theorem~\ref{thm:main} rules out any even-dimensional transfer theorem based only on Euclidean positive definiteness and a sufficiently small support radius. It does not conflict with positive results for structured subclasses, such as those obtained from P\'olya-type conditions or truncated powers \cite{BeatsonCastellXu2014,Lu2025,Xu2018}. Any positive even-dimensional result must therefore use additional information about the function, or replace direct geodesic substitution by a different distance transformation.
    \end{remark}

    \section*{Acknowledgements}

    The authors used OpenAI's ChatGPT and Codex to assist with English-language editing, preliminary numerical exploration of low-dimensional candidate counterexamples, certain auxiliary computations and integral estimates, and locating references to standard integral identities. The final theorem, construction, and proofs were developed and independently verified by the authors. All auxiliary calculations and integral estimates were checked independently, and all cited formulas and bibliographic information were verified against the original sources. The authors take full responsibility for the mathematical content and presentation of the paper.

	{\small
	\bibliographystyle{amsplain-full}
	\bibliography{references}
}

\end{document}